\documentclass[11pt,dvipsnames]{amsart}

\usepackage{style}

\usepackage{times}

\begin{document}

\newcommand{\Da}[1]{\textcolor{red}{(Da: #1)}}
\newcommand{\Gi}[1]{\textcolor{Green}{(Gi: #1)}}
\newcommand{\Ge}[1]{\textcolor{blue}{(Ge: #1)}}

\tikzstyle{nodal}=[circle,fill=black,inner sep=0pt, minimum width=4.5pt]
\tikzstyle{half-pencil}=[rectangle,draw=black,thick,inner sep=0pt, minimum width=5pt, minimum height=5pt]
\tikzset{double distance = 2pt}

\author{Gebhard Martin}
\address{Mathematisches Institut \\ Universität Bonn \\ Endenicher Allee 60 \\ 53115 Bonn \\ Germany}
\email{gmartin@math.uni-bonn.de} 

\author{Giacomo Mezzedimi}
\address{Institut für Algebraische Geometrie \\ Leibniz Universität Hannover \\ Welfengarten 1 \\ 30167 Hannover \\ Germany}
\email{mezzedimi@math.uni-hannover.de}

\author{Davide Cesare Veniani}
\address{Institut für Diskrete Strukturen und Symbolisches Rechnen \\ Universität Stuttgart \\ Pfaffenwaldring 57 \\ 70569 Stuttgart \\ Germany}
\email{davide.veniani@mathematik.uni-stuttgart.de}

\title{On extra-special Enriques surfaces}

\date{\today}
\subjclass[2020]{14J28 (14C20)}
\keywords{Enriques surface, genus one fibration, isotropic sequence}
\begin{abstract}
We refine Cossec and Dolgachev's classification of extra-special Enriques surfaces, providing a complete and concise proof.
\end{abstract}

\maketitle

\section{Introduction}
Throughout the paper, \(X\) will denote an Enriques surface defined over an algebraically closed field of arbitrary characteristic~\(p\).

A \emph{half-fiber} on~\(X\) is a divisor \(F\) such that \(2F\) is a fiber of a genus one fibration.
A \emph{\(c\)-sequence} on~\(X\) is a sequence of half-fibers \((F_1,\hdots,F_c)\) such that \(F_i.F_j = 1 - \delta_{ij}\).
As observed by Enriques himself for classical Enriques surfaces, and later extended to non-classical Enriques surfaces by Bombieri and Mumford \cite[Theorem~3]{Bombieri.Mumford:III}, every Enriques surface admits a \(1\)-sequence. The maximal length of a \(c\)-sequence is \(10\).

It is a natural question to ask whether every \(c\)-sequence can be extended to a \(c'\)-sequence with~\(c' > c\).
In this context, \emph{extra-special} Enriques surfaces play a fundamental role.

\begin{definition} \label{def: extra-special}
An Enriques surface \(X\) is called 
\begin{itemize}
    \item \emph{extra-special of type~\(\tilde{E}_8\)} if the dual graph of all \((-2)\)-curves on~\(X\) is
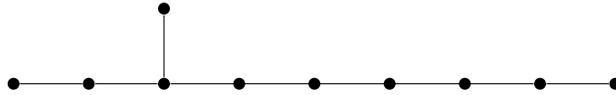
\[
\begin{tikzpicture}
    \node (R4) at (0,0) [nodal] {};
    \node (R5) at (0,1) [nodal] {};
    \node (R6) at (1,0) [nodal] {};
    \node (R7) at (2,0) [nodal] {};
    \node (R8) at (3,0) [nodal] {};
    \node (R9) at (4,0) [nodal] {};
    \node (RX) at (5,0) [nodal] {};
    \node (R3) at (-1,0) [nodal] {};
    \node (R2) at (-2,0) [nodal] {};
    \node (R11) at (6,0) [nodal] {};
    \draw (R2)--(R3)--(R4) (R5)--(R4)--(RX) (RX)--(R11);
\end{tikzpicture}
\]
\item  \emph{extra-special of type~\(\tilde{D}_8\)} if the dual graph of all \((-2)\)-curves on~\(X\) is
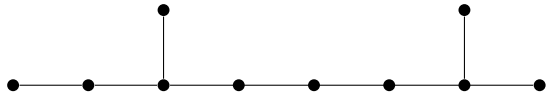
\[
\begin{tikzpicture}
    \node (R4) at (0,0) [nodal] {};
    \node (R5) at (0,1) [nodal] {};
    \node (R6) at (1,0) [nodal] {};
    \node (R7) at (2,0) [nodal] {};
    \node (R8) at (3,0) [nodal] {};
    \node (R9) at (4,0) [nodal] {};
    \node (RX) at (5,0) [nodal] {};
    \node (R3) at (-1,0) [nodal] {};
    \node (R2) at (-2,0) [nodal] {};
    \node (R1) at (4,1) [nodal] {};
\draw (R2)--(R3)--(R4) (R5)--(R4)--(RX) (R1)--(R9);
\end{tikzpicture}
\]
\item \emph{extra-special of type~\(\tilde{E}_7\)} if the dual graph of all \((-2)\)-curves on~\(X\) is
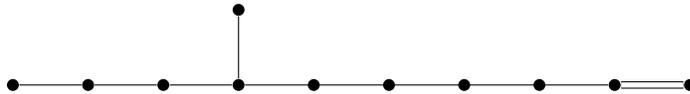
\[
\begin{tikzpicture}
    \node (R4) at (0,0) [nodal] {};
    \node (R5) at (0,1) [nodal] {};
    \node (R6) at (1,0) [nodal] {};
    \node (R7) at (2,0) [nodal] {};
    \node (R8) at (3,0) [nodal] {};
    \node (R9) at (4,0) [nodal] {};
    \node (RX) at (5,0) [nodal] {};
    \node (R3) at (-1,0) [nodal] {};
    \node (R2) at (-2,0) [nodal] {};
    \node (R1) at (-3,0) [nodal] {};
    \node (R11) at (6,0) [nodal] {};
    \draw (R1)--(R2)--(R3)--(R4) (R5)--(R4)--(RX);
    \draw[double] (RX)--(R11);
\end{tikzpicture}
\]
\end{itemize}
An Enriques surface is called \emph{extra-special} if it is extra-special of type~\(\tilde{E}_8\), \(\tilde{D}_8\), or \(\tilde{E}_7\).
\end{definition}

The classification of Enriques surfaces on which every \(1\)-sequence can be extended to a \(2\)-sequence is well-known and is contained in the following theorem due to Cossec and Dolgachev.

\begin{theorem}[{\cite[Theorem~3.4.1]{Cossec.Dolgachev}} or {\cite[Theorem~6.1.10]{CossecDolgachevLiedtke}}] \label{thm: non-degeneracy1}
If \(X\) is not extra-special of type~\(\tilde{E}_8\), then every \(1\)-sequence on~\(X\) extends to a \(2\)-sequence.
\end{theorem}

The main result of this paper is the following theorem, whose proof is obtained by combining \autoref{thm: directproof} and \autoref{cor: II^* fiber.consequences}.

\begin{theorem}\label{thm: main}
If \(X\) is not extra-special, then every \(2\)-sequence on~\(X\) extends to a \(3\)-sequence.
\end{theorem}

\begin{remark}
In their first book about Enriques surfaces, Cossec and Dolgachev proved that \(X\) admits a \(3\)-sequence if \(X\) is not extra-special \cite[Theorem~3.5.1]{Cossec.Dolgachev}. 
In a similar vein, Cossec proved that if \(p \neq 2\), then every \(1\)-sequence on \(X\) extends to a \(3\)-sequence \cite[Theorem~3.5]{Cossec:Picard_group}. Note that our result is strictly stronger than both of these results, since it asserts the extendability of any given \(2\)-sequence in all characteristics. Big parts of the proof of the result of Cossec are left to the `patient reader'. Moreover, the proof of the result of Cossec--Dolgachev occupies 32 pages and is described as `lengthy' in the new book on Enriques surfaces, with `no guarantee that it is correct' (see \cite[Remark 6.1.14 and Theorem~6.2.6]{DolgachevKondoBook}).
The main ingredient that allows us to give a substantially shorter proof of a stronger result is the observation that in a non-extendable \(2\)-sequence one of the genus one fibrations is special with a fiber of type~\(\II^*\) (\autoref{thm: directproof}), and that Enriques surfaces with such a fibration can be easily classified in all characteristics (\autoref{thm: II^* fiber}). 
\end{remark}

Extra-special Enriques surfaces do exist. Salomonsson~\cite{Salomonsson} gave equations for all extra-special surfaces of type~\(\tilde{E}_8\) and \(\tilde{E}_7\). An alternative construction of these surfaces as well as examples of extra-special surfaces of type~\(\tilde{D}_8\) were given by Katsura, Kond\={o} and Martin \cite[§§10, 11, 12]{Katsura.Kondo.Martin}.

There do exist non-extendable \(1\) or \(2\)-sequences on extra-special Enriques surfaces \cite[Proposition~6.2.7]{DolgachevKondoBook}. More precisely, on extra-special surfaces of type~\(\tilde{E}_8\) there exists only one genus one fibration, so the only \(1\)-sequence is non-extendable. On extra-special surfaces of type~\(\tilde{D}_8\) there are exactly three genus one fibrations, forming two distinct \(2\)-sequences, both non-extendable. Finally, on extra-special surfaces of type~\(\tilde{E}_7\) there are exactly two genus one fibrations, forming a non-extendable \(2\)-sequence.

\begin{remark} \label{rmk: extra-special}
Every extra-special surface admits a genus one fibration with a double fiber of additive type, and an extra-special surface of type~\(\tilde{E}_7\) even admits a quasi-elliptic fibration with two double fibers \cite[Proposition~6.2.7]{DolgachevKondoBook}. Consequently (cf. \autoref{lem: genus.1.fibrations.on.Enriques}), extra-special Enriques surfaces can exist only in characteristic \(2\) and they are either classical or supersingular. Additionally, extra-special surfaces of type~\(\tilde{E}_7\) can only be classical. 
\end{remark}

We infer an immediate corollary from \autoref{rmk: extra-special}.

\begin{corollary}
If one of the following conditions holds:
\begin{enumerate}
    \item \(p \neq 2\),
    \item \(p = 2\) and \(X\) is ordinary,
    \item \(p = 2\) and \(X\) is not extra-special,
\end{enumerate}
then every \(c\)-sequence on~\(X\) with \(c \leq 2\) extends to a \(3\)-sequence.
\end{corollary}

Let us describe one of the geometric consequences of the extendability of \(1\)- and \(2\)-sequences to \(3\)-sequences. By applying \cite[Theorem 3.5.1 and the following discussion]{CossecDolgachevLiedtke} and \cite[Section~7.8.1]{Cossec:Projective_models} to a \(3\)-sequence that extends a given \(2\)-sequence \((F_1,F_2)\), we obtain the following strong version of the Enriques--Artin Theorem for non-extra-special, classical Enriques surfaces.

\begin{corollary}
If \(X\) is a classical Enriques surface which is not extra-special, then \(X\) is the minimal resolution of a sextic \(S \subseteq \mathbb{P}^3\) given by an equation of the form 
\[
x_1x_2x_3 L Q + x_1^2x_2^2 L^2 + x_1^2x_3^2 L^2 + x_2^2x_3^2 L^2 + x_1^2x_2^2x_3^2,
\]
where \(L\) is linear and \(Q\) is a quadratic form. 
Moreover, for each half-fiber \(F_1\) (resp. \(2\)-sequence \((F_1,F_2)\)) on \(X\), there is a sextic model as above such that \(F_1\) (resp. \((F_1,F_2)\)) maps to a line (resp. a pair of lines) in the non-normal locus of \(S\). 
\end{corollary}

Extending \(3\)-sequences to \(4\)-sequences is a much more difficult problem, even in characteristic \(p \neq 2\). The classification of non-extendable \(3\)-sequences will be the subject of a forthcoming paper.

% On the one hand, there exist Enriques surfaces admitting a \(3\)-sequence that does not extend to a \(4\)-sequence (e.g. Enriques surfaces with finite automorphism group of type~\(\I\)). Nonetheless, in this example one can find other half-fibers forming a \(4\)-sequence.

% \begin{conjecture}
% If \(p \neq 2\), then every Enriques surface admits a \(4\)-sequence.
% \end{conjecture}

% This conjecture and the classification of non-extendable \(3\)-sequences will be the subject of a forthcoming paper

\section{Preliminaries}

In \autoref{sec: genus.one.fibrations}, we collect some basic definitions and facts about genus one fibrations on Enriques surfaces. 
In \autoref{sec: isotropic.sequences}, we introduce the notion of isotropic sequences, generalizing the notion of \(c\)-sequence given in the introduction.

\subsection{Genus one fibrations} \label{sec: genus.one.fibrations}
Recall that an \emph{Enriques surface} is a smooth and proper surface \(X\)of Kodaira dimension \(0\) with \(b_2(X) = 10\) over an algebraically closed field of arbitrary characteristic~\(p\).
An Enriques surface \(X\) is called \emph{classical} if its canonical divisor is not linearly equivalent to \(0\). A non-classical Enriques surface \(X\) (which can only exist if \(p = 2\)) is called \emph{ordinary} (or \emph{singular} or a \emph{\(\mathbf\mu_2\)-surface}) if the absolute Frobenius morphism is bijective on \(H^1(X,\mathcal O_X)\), and \emph{supersingular} (or an \emph{\(\mathbf{\alpha}_2\)-surface}) otherwise.

A \emph{genus one fibration} on \(X\) is a fibration \(f \colon X \rightarrow \IP^1\) such that the generic fiber \(X_\eta\) is a regular genus one curve. A genus one fibration is called \emph{elliptic} if \(X_\eta\) is smooth and \emph{quasi-elliptic} otherwise.
We will use Kodaira's notation for singular fibers. Singular fibers of type~\(\I_n\) are called of \emph{multiplicative} type, while all others are called of \emph{additive} type.
Any genus one fibration on \(X\) admits at least one \emph{half-fiber}, that is a curve \(F\) of arithmetic genus one such that \(2F\) is a fiber of the fibration. In particular, the degree of a multisection of a genus one fibration \(f\) on \(X\) is divisible by \(2\) and if \(f\) admits a \((-2)\)-curve as a bisection, both \(f\) and the bisection are called \emph{special}. We call a divisor \emph{primitive} if its class in \(\Num(X)\) spans a primitive sublattice.

\begin{lemma}[{\cite[Corollary~2.2.9]{CossecDolgachevLiedtke}}] \label{lem: half-fiber}
An effective divisor \(D\) on \(X\) is a half-fiber of a genus one fibration on~\(X\) if and only if \(D\) is nef, primitive, and \(D^2 = 0\).
\end{lemma}

We recall here the behaviour of genus one fibrations on Enriques surfaces.

\begin{lemma}[{\cite[Theorem~4.10.3]{CossecDolgachevLiedtke}}] \label{lem: genus.1.fibrations.on.Enriques}
Let \(f\colon X \rightarrow \IP^1\) be a genus one fibration on \(X\).
\begin{itemize}
    \item If \(p \neq 2\), then \(f\) is an elliptic fibration with two half-fibers, and each of them is either non-singular or singular of multiplicative type.
    \item If \(p = 2\) and \(X\) is classical, then \(f\) is an elliptic or quasi-elliptic fibration with two half-fibers, and each of them is either an ordinary elliptic curve or a singular curve of additive type.
    \item If \(p = 2\) and \(X\) is ordinary, then \(f\) is an elliptic fibration with one half-fiber, which is either a non-singular ordinary elliptic curve or a singular curve of multiplicative type.
    \item If \(p = 2\) and \(X\) is supersingular, then \(f\) is an elliptic or quasi-elliptic fibration with one half-fiber, which is either a supersingular elliptic curve or a singular curve of additive type. 
\end{itemize}
\end{lemma}

The symbol \(\sim\) denotes linear equivalence, and \(W_X^{\nod}\) denotes the Weyl group generated by reflections along classes of \((-2)\)-curves on~\(X\).

\begin{lemma}[{\cite[Theorem~3.2.1]{Cossec.Dolgachev}} or {\cite[Theorem~2.3.3]{CossecDolgachevLiedtke}}] \label{lem: reducibility}
If \(D\) is an effective divisor on \(X\) with \(D^2 \geq 0\), then there exist non-negative integers \(a_i\) and \((-2)\)-curves \(R_i\) such that
\[
    D \sim D' + \sum_i a_i R_i,
\]
where \(D'\) is the unique nef divisor in the \(W_X^{\nod}\)-orbit of \(D\). In particular, \(D^2 = D'^2\).
\end{lemma}

The following bound is an immediate consequence of the Shioda--Tate formula.

\begin{lemma} \label{lem: combinatorial.0}
The number of irreducible curves contained in \(s\) fibers of any genus one fibration on~\(X\) is at most \(8+s\). 
\end{lemma}

\subsection{Isotropic sequences} \label{sec: isotropic.sequences}
A \emph{\(c\)-degenerate (canonical isotropic) \(n\)-sequence on~\(X\)} is an \(n\)-tuple of the form
\[
    \big(F_1,F_1 + R_{1,1},\hdots,F_1 + \sum_{j=1}^{m_1} R_{1,j},F_2,F_2 + R_{2,1},\hdots,F_{c} + \sum_{j=1}^{m_c} R_{c,j}\big)
\]
where the \(F_i\) are half-fibers of genus one fibrations on~\(X\) and the \(R_{i,j}\) are \((-2)\)-curves satisfying the conditions
\begin{enumerate}
    \item \(F_i.F_j = 1 - \delta_{ij}\).
    \item \(R_{i,j}.R_{i,j+1} = 1\).
    \item \(R_{i,j}.R_{k,l} = 0\) unless \((i,j) = (k,l)\) or \((k,l)=(i,l+1)\).
    \item \(F_i.R_{i,1} = 1\) and \(F_i.R_{k,l} = 0\) if \((k,l) \neq (i,1)\).
\end{enumerate}
If \(c = n\), we simply call the above sequence a \emph{\(c\)-sequence}.

We say that a \(c\)-degenerate \(n\)-sequence \emph{extends} to a \(c'\)-degenerate \(n'\)-sequence if the former is contained in the latter, disregarding the ordering. The following fundamental theorem is due to Cossec and holds also in positive characteristic (cf. \cite[Proposition~6.1.7]{DolgachevKondoBook}).

\begin{theorem}[{\cite[Lemma~1.6.1, Theorem~3.3]{Cossec:Picard_group}}]  \label{thm: 10sequence}
If \(n \neq 9\), then every \(c\)-degenerate \(n\)-sequence on~\(X\) can be extended to a \(c'\)-degenerate \(10\)-sequence for some \(c' \geq c\).
\end{theorem}

A generic Enriques surface does not contain any \((-2)\)-curve. Thus, by \autoref{thm: 10sequence} every \(c\)-sequence with \(c \neq 9\) can be extended to a \(10\)-sequence. %In particular, a generic \(X\) contains a \(10\)-sequence, which leads to the \emph{Fano model} of \(X\), i.e. a realization of \(X\) as a scheme-theoretical intersection of degree \(10\) of \(10\) cubic hypersurfaces in \(\mathbb{P}^5\) (see \cite[§3.5]{CossecDolgachevLiedtke}).
% \footnote{\Da{In the bibliography of \cite[§3.5]{CossecDolgachevLiedtke} one also finds references to two original works by Fano. Shall we include them?} \Ge{Now that we have Corollary 1.7 in the introduction, I am not longer sure whether we have to include Fano models at all. But if we do, we should cite Fano's original work, I think.}\Da{I think this observation sounds now a bit superfluous, we can just omit it (I commented it)}}

\begin{lemma}[{\cite[Lemma~2.6.3]{CossecDolgachevLiedtke}} or {\cite[Lemma~3.5]{DolgachevMartin}}] \label{lem: new_combinatorialtwofibrations}
If \((F_1,F_2)\) is a \(2\)-sequence, then \(F_1\) and \(F_2\) do not have common irreducible components.
\end{lemma}

\section{Non-extendable \texorpdfstring{\(2\)}{2}-sequences}

In this section, we investigate the extendability of \(2\)-sequences and the properties of non-extendable \(2\)-sequences.

\begin{proposition} \label{prop: geometricargument}
Let \((F_1,F_2)\) be a \(2\)-sequence on an Enriques surface \(X\). If there are two simple fibers \(G_1 \in |2F_1|\) and \(G_2 \in |2F_2|\) with a common component, then there is a half-fiber \(F_3\) extending \((F_1,F_2)\) to a \(3\)-sequence.
\end{proposition}
\begin{proof}
By \cite[Proposition~2.6.1]{CossecDolgachevLiedtke}, the linear system \(|F_1 + F_2|\) is a pencil of curves without fixed components. Let \(R\) be a common component of \(G_1\) and \(G_2\). Then, \(R\) is a \((-2)\)-curve with \(R.(F_1 + F_2) = 0\). Hence, there exists a curve \(C \in |F_1 + F_2|\) containing \(R\). Set \(C' \coloneqq C - R\). Then, \(C'.F_i = (F_1+F_2-R).F_i = 1\) and \(C'^2 = (C - R)^2 = (F_1 + F_2)^2 - 2(F_1 + F_2).R + R^2 = 0\). Using \autoref{lem: reducibility}, we write \(C' = C'' + \sum_i a_iR_i\) with \(a_i \geq 0\), \(C''\) nef and \(C''^2 = C'^2=0\). Then, \(C''.F_i = C'.F_i - \sum a_jR_j.F_i \leq 1\).

Since \(C''\) is nef, \(C''.F_i=0\) for at most one \(i\) and then \(C''.F_j=1\) for \(j\ne i\). This implies that \(C''\) is primitive, and therefore \(C''\) is a half-fiber of some genus one fibration on \(X\) by \autoref{lem: half-fiber}. Hence, if \(C''.F_1 = C''.F_2 = 1\), then \((F_1,F_2,C'')\) is a \(3\)-sequence.

Thus, let us exclude the possibility that \(C''.F_1 = 0\) (the case \(C''.F_2 = 0\) is analogous). In this case, \(C''\) is a half-fiber of \(|2F_1|\). If \(C'' = F_1\), then
\[
    F_2 \sim (F_1 + F_2 - C'') \sim (F_1 + F_2 - C) + (R + \sum a_iR_i) \sim (R + \sum a_iR_i).
\]
Since \(|F_2|\) is \(0\)-dimensional, this shows that \(R\) is a component of \(F_2\). But this is impossible, since we assumed that \(R\) is a component of a simple fiber \(G_2 \in |2F_2|\). If \(C'' \neq F_1\), then the above argument applied to \(|F_2 + K_X|\) shows that \(R\) is a component of the other half-fiber of \(|2F_2|\), again contradicting the fact that \(R\) is contained in the simple fiber \(G_2 \in |2F_2|\).
\end{proof}

\begin{corollary} \label{claim:1}
If \((F_1,F_2)\) is a \(2\)-sequence on \(X\) that does not extend to a \(3\)-sequence, then one of the half-fibers of \(|2F_1|\) or \(|2F_2|\) is reducible and has at least \(4\) components.
\end{corollary}
\begin{proof}
By \autoref{thm: 10sequence}, we may assume that there exist \((-2)\)-curves \(R_{1,1},\hdots,R_{1,m},R_{2,1},\hdots,R_{2,n}\) with \(m+n = 8\) and such that \((F_1,F_1 + R_{1,1},\hdots,F_2,F_2 + R_{2,1},\hdots)\) is a \(2\)-degenerate \(10\)-sequence. Choose \(R\) to be one of the curves \(R_{i,j}\) with \(j \neq 1\). Since \(R.F_1 = R.F_2 = 0\), \(R\) is a component of two fibers \(G_1 \in |2F_1|\) and \(G_2 \in |2F_2|\). By \autoref{prop: geometricargument}, \(G_1\) and \(G_2\) cannot both be simple. Therefore, \(R\) is a component of a reducible half-fiber, which also contains all \(R_{i,j'}\) with \(j' \neq 1\). Since either \(m \geq 4\) or \(n \geq 4\), one of the reducible half-fibers has at least \(4\) components.
\end{proof}

The bound of \autoref{claim:1} on the number of components is not sharp, but it is enough to prove the following theorem, which is the key result of this paper.

\begin{theorem} \label{thm: directproof}
If \((F_1,F_2)\) is a \(2\)-sequence on \(X\) that does not extend to a \(3\)-sequence, then either \(|2F_1|\) or \(|2F_2|\) is a special genus one fibration with a fiber of type \(\II^*\).
\end{theorem}
\begin{proof}
By \autoref{claim:1} we can suppose that \(F_1\) is reducible and has at least \(4\) components. Let \(R_{2,1}\) be the component of \(F_1\) meeting \(F_2\). Since \(R_{2,1}\) is a simple component of \(F_1\), we can find two more components \(R_{2,2}\), \(R_{2,3}\) of \(F_1\) forming a chain with \(R_{2,1}\).

Applying \autoref{thm: 10sequence}, we extend the \(2\)-degenerate \(5\)-sequence \((F_1,F_2,\hdots,F_2+R_{2,1}+R_{2,2}+R_{2,3})\) to a \(2\)-degenerate \(10\)-sequence \((F_1,\hdots, F_1 + \sum_{i=1}^m R_{1,i},F_2,\hdots,F_2 + \sum_{i=1}^n R_{2,i})\) with \(n \geq 3\). Clearly, all the \(R_{2,i}\) are contained in \(F_1\).

Note that \(R_{2,1}\) is a special bisection of \(|2F_2|\). Since \(F_2.F_1 = F_2.R_{2,1} = 1\) and \(F_1 - R_{2,1}\) is connected (\(R_{2,1}\) being a simple component), all components of \(F_1\) except \(R_{2,1}\) are contained in a fiber \(G_2 \in |2F_2|\). We want to prove that \(G_2\) is of type \(\II^*\).

The fiber \(G_2\) is necessarily simple because of \autoref{lem: new_combinatorialtwofibrations}.
Let \(\Lambda\) be the sublattice of \(\Num(X)\) generated by the components of \(G_2\) and \(R_{2,1}\) and let \(\Lambda'=\Lambda[F_2]\) be the sublattice of \(\Num(X)\) obtained from \(\Lambda\) by adjoining \(F_2=\frac{1}{2}G_2\in \frac{1}{2}\Lambda\).

If \(m \geq 1\), then \(R_{1,1}\) meets a component of \(F_1\) distinct from \(R_{2,1}\), hence \(R_{1,1}\) meets a component of \(G_2\). But \(R_{1,1}.G_2 = 2R_{1,1}.F_2 = 0\), so \(R_{1,1}\) is contained in \(G_2\) and then so are all the \(R_{1,j}\). This shows that \(\Lambda\) has rank \(10\), because it contains all \(R_{i,j}\), the class of \(F_1\) and the class of \(G_2 \equiv 2F_2\).
Since \(\Lambda'\) contains all the \(10\) divisors of the \(10\)-sequence and since these divisors generate a lattice of rank \(10\) and discriminant \(9\), \(\Lambda'\) has index \(1\) or \(3\) in \(\Num(X)\).

By \cite[Lemma 1.6.2]{Cossec:Picard_group}, there exists a vector \(e \in \Num(X)\) with \(e^2 = 0, e.F_1 = e.F_2 = 1\), \(e.R_{i,j} = 0\) for all \((i,j)\) except for \(e.R_{2,n-1} = 1\), and such that \(e\) and the components of the \(10\)-sequence generate \(\Num(X)\). (In the notation of \cite{Cossec:Picard_group}, we can choose \(e = e_{9,10}\), since \(n \geq 3\).)

Note that it suffices to show that \(e\) is contained in the sublattice \(\Lambda'\). Indeed, if this holds, then \(\Lambda'=\Num(X)\) and therefore \(\Lambda\) has index at most \(2\) in \(\Num(X)\). Then, consider the basis of \(\Lambda\) given by \(G_2, R_{2,1}\) and all the components of \(G_2\) except a simple one. The intersection matrix of \(\Lambda\) with respect to this basis is
\[
    \left(\begin{array}{c c |c}
    0 &2 &0\\
    2 & -2 &*\\
    \hline
    0 &* &L
    \end{array}\right),
\]
where \(L\) is a root lattice of rank \(8\) depending on the type of \(G_2\). More precisely, \(L\) is the root lattice associated to the Dynkin diagram obtained by removing a simple component from \(G_2\). By reducing along the first row and first column, we obtain that \(\det(\Lambda)=-4\det(L)\). Since \(\Lambda\) has index at most~\(2\) in the unimodular lattice \(\Num(X)\), we have \(|{\det(\Lambda)}|\le 4\), and hence \(\det(L)=1\). This forces \(L\) to be isometric to the lattice \(E_8\), and in turn \(G_2\) to be of type~\(\II^*\).

It remains to show that \(e \in \Lambda'\). Let \(D\) be an effective lift of \(e\) to \(\Pic(X)\). 
Write \(D \sim D' + \sum_i a_iR_i\) as in \autoref{lem: reducibility}. 
Since \(D'^2 = D^2 = 0\) and \(D'.F_i \leq 1\) with equality for at least one \(i\), \(D'\) is nef and primitive, hence equal to a half-fiber by \autoref{lem: half-fiber}.
If \(D'\) is not a half-fiber of \(|2F_1|\) or \(|2F_2|\), then \((F_1,F_2,D')\) is a \(3\)-sequence, contradicting our assumption that \((F_1,F_2)\) does not extend to a \(3\)-sequence. 
Therefore, we can suppose that \(D'\) is a half-fiber of \(|2F_k|\) for some \(k \in \{1,2\}\). 
Take \(j \in \{1,2\}\) with \(j \neq k\). 
Then, \(D'.F_k = 0\) and \(D'.F_j = 1\), so \((\sum_i a_iR_i)^2 = (D - D')^2 = -2\) and \((\sum_i a_iR_i).F_j = (D - D').F_j = 0\). 
Thus, \(\sum_i a_iR_i\) is a connected configuration of \((-2)\)-curves contained in a single fiber of \(|2F_j|\). 
But \((\sum_i a_iR_i).R_{2,n-1} = (D - D').R_{2,n-1} = D.R_{2,n-1} = 1\), since \(R_{2,n-1}.F_1 = R_{2,n-1}.F_2 = 0\), hence \(\sum_i a_iR_i\) is contained in the fiber of \(|2F_j|\) containing \(R_{2,n-1}\).
If \(j = 1\), this fiber is \(F_1\), whereas if \(j = 2\) this fiber is \(G_2\). In both cases, we obtain that \(e \in \Lambda'\).
\end{proof}

\section{Special genus one fibrations with a fiber of type~\texorpdfstring{\(\II^*\)}{II*}}
By \autoref{thm: directproof}, one of the two fibrations in a non-extendable \(2\)-sequence is special with a fiber of type~\(\II^*\). It turns out that admitting such a fibration is such a restrictive property that all Enriques surfaces satisfying this property can be classified. We will carry out this classification in this section, thereby proving \autoref{thm: main} (see \autoref{cor: II^* fiber.consequences}).

\begin{definition}
An Enriques surface~\(X\) is called
\begin{enumerate}
    \item \emph{of type~\(\I\)} if the dual graph of all \((-2)\)-curves on~\(X\) is
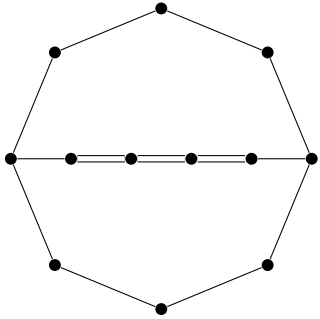
\[
\begin{tikzpicture}
\node (R24) at (0:2) [nodal] {};
\node (R23) at (45:2) [nodal] {};
\node (R22) at (90:2) [nodal] {};
\node (R21) at (135:2) [nodal] {};
\node (R8) at (180:2) [nodal] {};
\node (R11) at (180:1.2) [nodal] {};
\node (A) at (180:0.4) [nodal] {};
\node (B) at (0:0.4) [nodal] {};
\node (C) at (0:1.2) [nodal] {};
\node (R27) at (225:2) [nodal] {};
\node (R26) at (270:2) [nodal] {};
\node (R25) at (315:2) [nodal] {};
\draw (R8)--(R21)--(R22)--(R23)--(R24)--(R25)--(R26)--(R27)--(R8)--(R11) (R24)--(C);
\draw[double] (C)--(B) (B)--(A) (A)--(R11);
\end{tikzpicture}
\]
    \item \emph{of type~\(\tilde{E}_7^{(2)}\)} if the dual graph of all \((-2)\)-curves on~\(X\) is
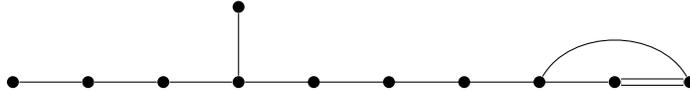
\[
\begin{tikzpicture}
    \node (R4) at (0,0) [nodal] {};
    \node (R5) at (0,1) [nodal] {};
    \node (R6) at (1,0) [nodal] {};
    \node (R7) at (2,0) [nodal] {};
    \node (R8) at (3,0) [nodal] {};
    \node (R9) at (4,0) [nodal] {};
    \node (RX) at (5,0) [nodal] {};
    \node (R3) at (-1,0) [nodal] {};
    \node (R2) at (-2,0) [nodal] {};
    \node (R1) at (-3,0) [nodal] {};
    \node (R11) at (6,0) [nodal] {};
    \draw (R1)--(R2)--(R3)--(R4) (R5)--(R4)--(RX);
    \draw[double] (RX)--(R11);
    \draw (R9) to[bend left=60] (R11);
    \end{tikzpicture}
    \]
\end{enumerate}
\end{definition}

\begin{theorem} \label{thm: II^* fiber}
If an Enriques surface \(X\) admits a special genus one fibration with a fiber of type~\(\mathrm{II}^*\), then one of the following holds:
\begin{enumerate}
     \item \(X\) is of type~\(\I\),
    \item \(X\) is of type~\(\tilde{E}_7^{(2)}\),
    \item \(X\) is extra-special.
\end{enumerate}
\end{theorem}
\begin{proof}
First, note that it suffices to show that \(X\) contains a configuration of \((-2)\)-curves equal to the dual graph of one of the five types of surfaces in the theorem. Indeed, by \cite[Theorem~2.3]{Vinberg} and \cite[Remark~2.15]{Katsura.Kondo.Martin}, the group generated by reflections along the \((-2)\)-curves in the dual graph of one of these surfaces has finite index in the orthogonal group of \(\Num(X)\). By \cite[Proposition~6.9]{Namikawa}, this implies that the \((-2)\)-curves in the graph are in fact all the \((-2)\)-curves on a surface containing such a configuration.

Assume first that the fiber of type~\(\II^*\) is a half-fiber. 
Its special bisection must meet the only component of the half-fiber of multiplicity \(1\), giving rise to the dual graph of type~\(\tilde{E}_8\).

Assume instead that the fiber of type~\(\II^*\) is simple. Since this fiber has only one simple component, the special bisection \(R\) intersects transversely one of the two components of multiplicity \(2\) or it intersects doubly the component of multiplicity \(1\). 
If \(R\) intersects the component of multiplicity \(2\) next to the simple component, then we obtain the dual graph of type~\(\tilde{D}_8\). 
If \(R\) intersects the other component of multiplicity \(2\), then we obtain a new half-fiber of type~\(\mathrm{III}^*\):
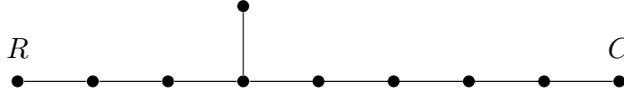
\[
\begin{tikzpicture}
    \node (R4) at (0,0) [nodal] {};
    \node (R5) at (0,1) [nodal] {};
    \node (R6) at (1,0) [nodal] {};
    \node (R7) at (2,0) [nodal] {};
    \node (R8) at (3,0) [nodal] {};
    \node (R9) at (4,0) [nodal] {};
    \node[label=above:\(C\)] (RX) at (5,0) [nodal] {};
    \node (R3) at (-1,0) [nodal] {};
    \node (R2) at (-2,0) [nodal] {};
    \node[label=above:\(R\)] (R1) at (-3,0) [nodal] {};
    \draw (R1)--(R2)--(R3)--(R4) (R5)--(R4)--(RX);
    \end{tikzpicture}
    \]
The component \(C\) is orthogonal to the fiber of type~\(\mathrm{III}^*\), hence it belongs to another reducible fiber~\(G\) of this new fibration. By \autoref{lem: combinatorial.0}, \(G\) must be of type~\(\I_2\) or \(\III\). If \(G\) is a half-fiber, then we obtain the dual graph of type~\(\tilde{E}_7\). 
If instead it is a simple fiber, we obtain the dual graph of type~\(\tilde{E}_7^{(2)}\). 
Finally, if \(R\) intersects doubly the simple component of the fiber of type~\(\II^*\), then we obtain a new half-fiber \(G\) of type~\(\I_2\) or \(\III\). The seven components orthogonal to \(G\) are contained in a fiber of type~\(\mathrm{III}^*\) by \autoref{lem: combinatorial.0}, which can be a half-fiber or a simple fiber. 
In the former case, we obtain the dual graph of type~\(\tilde{E}_7\). 
In the latter, the bisection of the \(\mathrm{III}^*\) fiber also intersects the other simple component:
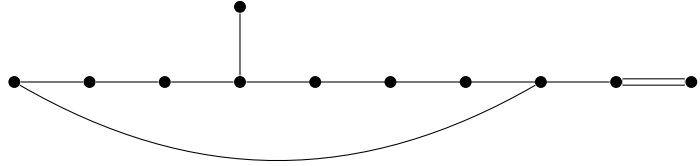
\[
\begin{tikzpicture}
    \node (R4) at (0,0) [nodal] {};
    \node (R5) at (0,1) [nodal] {};
    \node (R6) at (1,0) [nodal] {};
    \node (R7) at (2,0) [nodal] {};
    \node (R8) at (3,0) [nodal] {};
    \node (R9) at (4,0) [nodal] {};
    \node (RX) at (5,0) [nodal] {};
    \node (R3) at (-1,0) [nodal] {};
    \node (R2) at (-2,0) [nodal] {};
    \node (R1) at (-3,0) [nodal] {};
    \node (R11) at (6,0) [nodal] {};
    \draw (R1)--(R2)--(R3)--(R4) (R5)--(R4)--(RX);
    \draw[double] (RX)--(R11);
    \draw (R9) to[bend left] (R1);
\end{tikzpicture}
\]
There is a half-fiber of type~\(\I_8\), hence, by \autoref{lem: genus.1.fibrations.on.Enriques}, either \(p \neq 2\) or \(X\) is ordinary. Thus, it follows from \cite[Theorem 3.1]{Martin} that \(X\) is of type~\(\I\).
\end{proof}

\begin{corollary} \label{cor: II^* fiber.consequences}
If \(X\) is an Enriques surfaces with a special genus one fibration with a fiber of type~\(\II^*\), then \(X\) satisfies the following properties:
\begin{enumerate}
    \item \label{claim: cor_1} \(X\) contains finitely many \((-2)\)-curves,
    \item \label{claim: cor_2} \(X\) has finite automorphism group,
    \item \label{claim: cor_3} if \(p \neq 2\) or \(X\) is ordinary, then \(X\) is of type~\(\I\),
    \item \label{claim: cor_4} \(X\) is not extra-special if and only if every \(c\)-sequence with \(c \le 2\) extends to a \(3\)-sequence.
\end{enumerate}
\end{corollary}

\begin{proof}
Claim \eqref{claim: cor_1} follows from the classification in \autoref{thm: II^* fiber}, since all the surfaces listed there contain only finitely many \((-2)\)-curves.

By the first paragraph of the proof of \autoref{thm: II^* fiber}, the Weyl group \(W_X^{\nod}\) has finite index in the orthogonal group of \(\Num(X)\). Since \(\Aut(X)\) acts with finite kernel on \(\Num(X)/W_X^{\nod}\) by \cite[Theorem]{DolgachevMartin}, this implies Claim~\eqref{claim: cor_2} (see also \cite[Corollary 3.4]{Dolgachev1984}). 

For Claim~\eqref{claim: cor_3}, by \autoref{lem: genus.1.fibrations.on.Enriques} it suffices to observe that if \(X\) is not of type~\(\I\), then there exists an additive half-fiber on \(X\).

Finally, for Claim~\eqref{claim: cor_4}, one can simply list all genus one fibrations on \(X\). For the surfaces of type~\(\I\), we refer to \cite[Proposition~8.9.6]{DolgachevKondoBook}. For the extra-special surfaces and for the surfaces of type~\(\tilde{E}_7^{(2)}\), we refer to \cite[Proposition~6.2.5]{DolgachevKondoBook}.
\end{proof}

\begin{remark}
Enriques surfaces of type~\(\I\) and of type~\(\tilde{E}_7^{(2)}\) also display a special behaviour with respect to \(c\)-sequences: on both types of surfaces, there exist \(3\)-sequences that cannot be extended to \(4\)-sequences. In fact, there are no \(4\)-sequences at all on type~\(\tilde{E}_7^{(2)}\) \cite[Proposition~6.2.5]{DolgachevKondoBook}.
\end{remark}

\bibliographystyle{amsplain}
\bibliography{Enriques}

\providecommand{\bysame}{\leavevmode\hbox to3em{\hrulefill}\thinspace}
\providecommand{\MR}{\relax\ifhmode\unskip\space\fi MR }
\providecommand{\MRhref}[2]{%
  \href{http://www.ams.org/mathscinet-getitem?mr=#1}{#2}
}
\providecommand{\href}[2]{#2}
\begin{thebibliography}{10}

\bibitem{Bombieri.Mumford:III}
Enrico Bombieri and David Mumford, \emph{Enriques' classification of surfaces
  in char. {$p$}. {III}}, Invent. Math. \textbf{35} (1976), 197--232.
  \MR{491720}

\bibitem{Cossec:Picard_group}
Fran\c{c}ois Cossec, \emph{On the {P}icard group of {E}nriques surfaces}, Math.
  Ann. \textbf{271} (1985), no.~4, 577--600. \MR{790116}

\bibitem{Cossec.Dolgachev}
Fran\c{c}ois Cossec and Igor Dolgachev, \emph{\rm{Enriques surfaces. {I}}},
  Progress in Mathematics, vol.~76, Birkh\"auser Boston, Inc., Boston, MA,
  1989. \MR{986969}

\bibitem{CossecDolgachevLiedtke}
Fran\c{c}ois Cossec, Igor Dolgachev, and Christian Liedtke, \emph{Enriques
  surfaces {I}}, with an appendix by Shigeyuki Kond\={o}, draft available on
  \url{www.math.lsa.umich.edu/~idolga/EnriquesOne.pdf}, 2021.

\bibitem{Cossec:Projective_models}
Fran\c{c}ois~R. Cossec, \emph{Projective models of {E}nriques surfaces}, Math.
  Ann. \textbf{265} (1983), no.~3, 283--334. \MR{721398}

\bibitem{Dolgachev1984}
Igor Dolgachev, \emph{On automorphisms of {E}nriques surfaces}, Invent. Math.
  \textbf{76} (1984), no.~1, 163--177. \MR{739632}

\bibitem{DolgachevKondoBook}
Igor Dolgachev and Shigeyuki Kond\=o, \emph{Enriques surfaces {II}}, draft
  available on \url{www.math.lsa.umich.edu/~idolga/EnriquesTwo.pdf}, 2021.

\bibitem{DolgachevMartin}
Igor Dolgachev and Gebhard Martin, \emph{Numerically trivial automorphisms of
  {E}nriques surfaces in characteristic 2}, J. Math. Soc. Japan \textbf{71}
  (2019), no.~4, 1181--1200. \MR{4023303}

\bibitem{Katsura.Kondo.Martin}
Toshiyuki Katsura, Shigeyuki Kond\={o}, and Gebhard Martin,
  \emph{Classification of {E}nriques surfaces with finite automorphism group in
  characteristic 2}, Algebr. Geom. \textbf{7} (2020), no.~4, 390--459.
  \MR{4156410}

\bibitem{Martin}
Gebhard Martin, \emph{Enriques surfaces with finite automorphism group in
  positive characteristic}, Algebr. Geom. \textbf{6} (2019), no.~5, 592--649.
  \MR{4009175}

\bibitem{Namikawa}
Yukihiko Namikawa, \emph{Periods of {E}nriques surfaces}, Math. Ann.
  \textbf{270} (1985), no.~2, 201--222. \MR{771979}

\bibitem{Salomonsson}
Pelle Salomonsson, \emph{Equations for some very special {E}nriques surfaces in
  characteristic two}, preprint,
  \href{https://arxiv.org/abs/math/0309210}{arXiv:math/0309210}, 2003.

\bibitem{Vinberg}
\`Ernest~B. Vinberg, \emph{Some arithmetical discrete groups in {L}oba\v
  cevski\u\i \ spaces}, Discrete subgroups of {L}ie groups and applications to
  moduli ({I}nternat. {C}olloq., {B}ombay, 1973), Oxford Univ. Press, Bombay,
  1975, pp.~323--348. \MR{0422505}

\end{thebibliography}

\end{document}